\documentclass{amsart}

\usepackage{amsmath, amssymb, amsthm}
\usepackage[all]{xy}

\newtheorem{thm}{Theorem}[section]
\newtheorem{prop}[thm]{Proposition}
\newtheorem{lem}[thm]{Lemma}
\newtheorem{cor}[thm]{Corollary}

\DeclareMathOperator{\kernel}{Ker}

\DeclareMathOperator{\Aut}{Aut}

\begin{document}

\title[Loewy lengths of centers]{Some studies on Loewy lengths \\ of centers of $p$-blocks}
\author[Y. Otokita]{Yoshihiro Otokita}
\address{Yoshihiro Otokita: \newline
Department of Mathematics and Informatics, \newline
Graduate School of Science, \newline
Chiba University, \newline
1--33 Yayoi-Cho, Inage-Ku, \newline
Chiba--Shi, 263--8522, Japan.}
\email{otokita@chiba-u.jp}

\maketitle

\begin{abstract}
In this paper, we investigate some relations between the Loewy lengths of the centers of blocks of group algebras and its defect groups. In particular, we give a new upper bound of the Loewy length and determine the structure of blocks with large Loewy length. 
\end{abstract}

\section{Introduction}
Let $p$ be a prime, $G$ a finite group and $(K, \mathcal{O}, F)$ a splitting $p$-modular system for $G$ where $\mathcal{O}$ is a complete discrete valuation ring with quotient field $K$ of characteristic $0$ and residue field $F$ of characteristic $p$. Passman has shown an upper bound for the Loewy length $LL({\bf Z}FG)$ of the center of group algebra $FG$. 

\begin{thm} [Passman \cite{P}]\label{Thm1.1} 
\[ LL({\bf Z}FG) \le (p^{a+1}-1)/(p-1) \]
where $|G|=p^{a}m, p \nmid m$.
\end{thm}

In the following, let $B$ be a block of $FG$ with a defect group $D$. Okuyama has generalized Theorem \ref{Thm1.1} for the center ${\bf Z}B$ of $B$.

\begin{thm} [Okuyama \cite{O}]\label{Thm1.2} 
\[ LL({\bf Z}B) \le |D| \]
with equality if and only if $D$ is cyclic and $B$ is nilpotent.
\end{thm}

In this paper, we investigate some relations between $B, D$ and $LL({\bf Z}B)$. In the next section, we show some fundamental properties and some examples of them. The third section deals with a new upper bound below in terms of the order and exponent of $D$.

\begin{thm} \label{Thm1.3} Let $p^{d}$ and $p^{\varepsilon}$ be the order and exponent of $D$, respectively. Then we have
\[ LL({\bf Z}B) \le p^{d} - p^{d - \varepsilon} +1. \]
\end{thm}
 
In the last section, we determine the structure of blocks with large Loewy length as follows
(In the following, we denote by $C_{m}$ a cyclic group of order $m$ and by $C_{m} \times C_{n}$ a direct product of two cyclic groups).

\begin{thm} \label{Thm1.4} If $LL({\bf Z}B) = |D| - 1$, then one of the following holds:
\begin{enumerate}
\item $D \simeq C_{3}$.
\item $D \simeq C_{2} \times C_{2}$ and $B$ is Morita equivalent to $FD$. 
\end{enumerate}
\end{thm}

\begin{thm} \label{Thm1.5} If $LL({\bf Z}B) = |D| - 2$, then one of the following holds:
\begin{enumerate}
\item $D \simeq C_{5}$.
\item $D \simeq C_{2} \times C_{2}$ and $B$ is Morita equivalent to $FA_{4}$.
\item $D \simeq C_{2} \times C_{2}$ and $B$ is Morita equivalent to the principal block of $FA_{5}$.
\end{enumerate}
where $A_{4}$ and $A_{5}$ are four and five degree alternating groups, respectively.
\end{thm}

\begin{thm} \label{Thm1.6} If $LL({\bf Z}B)=|D|-3$, then one of the following holds:
\begin{enumerate}
\item $D \simeq C_{5}$.
\item $D \simeq C_{7}$.
\item $D \simeq C_{4} \times C_{2}$ and $B$ is Morita equivalent to $FD$.
\end{enumerate}
\end{thm}

By Theorem \ref{Thm1.2} and the results of Brou\'{e}-Puig \cite{BP} and Puig \cite{Pu}, if $LL({\bf Z}B) = |D|$, then $B$ is Morita equivalent to $FD$ where $D$ is cyclic. So we have determined the structure of blocks with $|D|-3 \le LL({\bf Z}B) \le |D|$.  

\section{Some fundamental results}
In this paper, we denote by $k(B)$ and $l(B)$ the numbers of irreducible ordinary and Brauer characters associated to $B$, respectively. The next proposition is clear by the fact that $Z(B)$ is local.
\begin{prop} \label{Prop2.1} The following are equivalent:
\begin{enumerate}
\item $D$ is trivial.
\item $LL({\bf Z}B) = 1$.
\end{enumerate}
\end{prop}

Moreover, we give an upper bound of $LL({\bf Z}B)$ in terms of $k(B)$ and $l(B)$.

\begin{prop} \label{Prop2.2} 
\[ LL({\bf Z}B) \le k(B) - l(B) + 1.\]
\end{prop}

\begin{proof}
Let denote by ${\bf S}B$ and ${\bf SZ}B$ the socle of $B$ and ${\bf Z}B$, respectively. Then $k(B) = \dim Z(B)$, $l(B) = \dim {\bf S}B \cap {\bf Z}B$ and ${\bf S}B \cap {\bf Z}B \subseteq {\bf SZ}B$ in general. Thus we have
\begin{align*}
LL({\bf Z}B) - 1 & \le \dim {\bf Z}B - \dim {\bf SZ}B \\
                         & \le \dim {\bf Z}B - \dim {\bf S}B \cap {\bf Z}B\\
                         & = k(B) - l(B),
 \end{align*}
 as required.
 \end{proof}
   
In the following, let $\beta$ be a root of $B$, that is, a block of $F[DC_{G}(D)]$ such that $\beta^{G} = B$. We denote by $N_{G}(D, \beta)$ the inertial group of $\beta$ in $N_{G}(D)$, by $I(B)$ the inertial quotient $N_{G}(D, \beta) / DC_{G}(D)$ and by $e(B) = |I(B)|$ the inertial index of $B$. In case $D$ is cyclic, the Loewy length is given in \cite{KKS}.

\begin{prop} [{\cite[Corollary 2.8]{KKS}}] \label{Prop2.3} If $D$ is cyclic, then
\[ LL({\bf Z}B) = \frac{|D| - 1}{e(B)} + 1. \]
\end{prop}

In the following, ${\bf J}A$ denotes the Jacobson radical of an algebra $A$ over $F$. The remainder of this section is devoted to some blocks with abelian defect group $D \simeq C_{p^{m}} \times C_{p^{n}}$ for some $m, n \ge 1$. These results are applied to the proof of our main theorems in the last section. First of all, we show the next lemma.

\begin{lem} \label{Lem2.4} If $D$ is normal in $G$, then $LL({\bf Z}B) \le p^{m} + p^{n} - 1$. In particular, if $B$ is perfect isometric to its Brauer correspondent in $N_{G}(D)$, then $LL({\bf Z}B) \le p^{m} + p^{n} - 1$.
\end{lem}
\begin{proof}
By a result of K\"{u}lshammer \cite{K}, $B$ is Morita equivalent to $F^{\alpha}[D \rtimes I(B)]$ where $\alpha$ is a $2$-cocycle in $D \rtimes I(B)$. Hence ${\bf Z}B \simeq {\bf Z}F^{\alpha}[D \rtimes I(B)]$ as algebra and 
\begin{align*}
LL({\bf Z}B)  & = LL({\bf Z}F^{\alpha}[D \rtimes I(B)]) \\
                    & \le LL(F^{\alpha}[D \rtimes I(B)]). 
 \end{align*}
 By Lemma 1.2, Proposition 1.5 and Lemma 2.1 in \cite{P2}, ${\bf J}F^{\alpha}[D \rtimes I(B)] = {\bf J}FD \cdot F^{\alpha}[D \rtimes I(B)] = F^{\alpha} [D \rtimes I(B)] \cdot {\bf J}FD$ and thus $LL(F^{\alpha}[D \rtimes I(B)]) = LL(FD)$. Moreover, by Theorem (3) in \cite{M}, we have $LL(FD) = p^{m} + p^{n} -1$, as claimed. 
 The second part of our claim is clear by the first part.
 \end{proof}
 
 Now we consider the case $p=2$.

\begin{prop} \label{Prop2.5} Assume $D \simeq C_{2^{m}} \times C_{2^{n}}$ for some $m, n \ge 1$. Then the one of the following holds:
\begin{enumerate}
\item $B$ is Morita equivalent to $FD$ and $LL({\bf Z}B) = 2^{m} + 2^{n} -1$.
\item $B$ is Morita equivalent to $FA_{4}$ and $LL({\bf Z}B)=2$.
\item $B$ is Morita equivalent to the principal block of $FA_{5}$ and $LL({\bf Z}B)=2$.
\item $B$ is Morita equivalent to $F[D \rtimes C_{3}]$ and $LL({\bf Z}B) \le 2^{m+1} - 1$. (In this case $2 \le m = n$)
\end{enumerate}
where $A_{4}$ and $A_{5}$ are four and five degree alternating groups, respectively.
\end{prop}

\begin{proof}
Without loss of generality, we may assume $m \ge n$.
We first obtain the order of automorphsim group $\Aut (D)$ of $D$ as follows.
\begin{equation*}
 |\Aut (D)| = \begin{cases}
                        3 \cdot 2^{4m-3} & \text{if $m = n$} \\
                            2^{m + 3n -2} & \text{if $m > n$}.
                    \end{cases}
 \end{equation*}
 
 We remark that $e(B)$ divides the odd part of $|\Aut (D)|$ and investigate the following cases. \\
 
 \textbf{Case 1} $m > n$. \\
 We have $e(B)=1$ and thus $B$ is Morita equivalent to $FD$ by \cite{Pu}. Since ${\bf Z}B \simeq {\bf Z}FD$ as algebra, we deduce $LL({\bf Z}B) = LL({\bf Z}FD) = LL(FD) = 2^{m} + 2^{n} - 1$ by \cite{M}. \\

 \textbf{Case 2} $m = n = 1$. \\
 By the result of Erdmann \cite{E}, $B$ is Morita equivalent to $FD$, or $FA_{4}$ or the principal block of $FA_{5}$. In the first case, we have $LL({\bf Z}B) = 3$ as same way to \textbf{Case 1}. In the remaining cases, $LL({\bf Z}B)=2$ by using Proposition \ref{Prop2.2} since $k(B) - l(B)=1$. \\
 
 \textbf{case 3} $m = n \ge 2$. \\
 By Theorem 1.1 in \cite{EKKS} , $B$ is Morita equivalent to $FD$ or $F[D \rtimes C_{3}]$. So we have $LL({\bf Z}B) = 2^{m+1} - 1$ by the same way to \textbf{Case 1} or $LL({\bf Z}B) \le 2^{m+1} - 1$ by Lemma \ref{Lem2.4}, respectively.
 
 \end{proof}
         
At the end of this section, we study the case that $D \simeq C_{3} \times C_{3^n}$ for some $n \ge 1$.

\begin{prop} \label{Prop2.6} If $D \simeq C_{3} \times C_{3^{n}}$ for some $n \ge 1$, then $LL({\bf Z}B) \le 3^{n} + 2$. In particular, $LL({\bf Z}B) \le |D| - 4$.
\end{prop}
 
\begin{proof}
We first obtain
\begin{equation*}
 |\Aut (D)| = \begin{cases}
                        16 \cdot 3 & \text{if $n=1$} \\
                            4 \cdot 3^{n+1} & \text{if $n \ge 2$}.
                    \end{cases}
 \end{equation*}
 
 \textbf{Case 1} $e(B) \le 4$. \\
 By results of \cite{Pu}, \cite{PU1}, \cite{PU2}, \cite{U} and Lemma \ref{Lem2.4}, we deduce $LL({\bf Z}B) \le 3^{n} + 2$. \\
 
 Since $e(B)$ divides $3'$-part of $|\Aut (D)|$, we may assume $n=1$ in the following. \\
 
 \textbf{Case 2} $n=1$ and $5 \le e(B)$. \\
$I(B)$ is isomorphic to one of the following groups:

\begin{itemize}
\item $C_{8}$, $D_{8}$ (dihedral group of order $8$), $Q_{8}$ (quaternion group of order $8$),
\item $SD_{16}$ (semi-dihedral group of order $16$).
\end{itemize}

We first suppose $I(B)$ is isomorphic to $D_{8}$ or $SD_{16}$. By the results of Kiyota \cite{Ki} and Watanabe \cite{W}, $k(B) - l(B)$ is at most $4$ and thus $LL({\bf Z}B) \le 5$ by Proposition \ref{Prop2.2}. Finally, suppose $I(B)$ is isomorphic to $C_{8}$ or $Q_{8}$. Kiyota \cite{Ki} has not determined the invariants $k(B), l(B)$ in general. However, we can compute $k(B) - l(B)$ as follows. Since $I(B)$ acts on $D \backslash \{1\}$ regularly, the conjugacy classes for $B$-subsections are $(1, B)$ and $(u, b_{u})$ for some $u \in D \backslash \{1\}$ where $b_{u}$ is a block of $FC_{G}(u)$ such that $(b_{u})^{G}=B$. Moreover $I(b_{u}) \simeq C_{I(B)}(u)$ is trivial by the action of $I(B)$ on $D \backslash \{ 1 \}$, $b_{u}$ is nilpotent, $k(B) - l(B) = l(b_{u}) = 1$, and hence $LL({\bf Z}B) = 2$ as claimed. \\

The last part of the proposition is clear.
  
\end{proof}

\section{Proof of Theorem \ref{Thm1.3}}
We describe some notations to prove Theorem \ref{Thm1.3}. For a $p$-element $x$ in $G$, we denote by 
\begin{align*}
& \sigma_{x} : {\bf Z}FG \to {\bf Z}FC_{G}(x), \\
& \tau_{x} : FC_{G}(x) \to F[C_{G}(x) / \langle x \rangle]
\end{align*}
the Brauer homomorphism and natural epimorphism, respectively. When $\sigma_{x} (1_{B}) \neq 0$ where $1_{B}$ is the block idempotent of $B$, let $b_{1}, \dots, b_{r}$ be all blocks of $FC_{G}(x)$ such that $\sigma_{x}(1_{B})1_{b_{i}} \neq 0$. For each $1 \le i \le r$, $\tau_{x}(1_{b_{i}})$ is the unique block idempotent of $F[C_{G}(x) / \langle x \rangle]$. We put $\bar{b_{i}} = F[C_{G}(x) / \langle x \rangle] \tau_{x}(1_{b_{i}})$. Now we define two integers as follows.

\begin{align*}
& \lambda_{x} = \max \{ LL({\bf Z}\bar{b_{i}}) \ | \ 1 \le i \le r \}, \\
& \lambda = \max \{ \lambda_{x} (|x| - 1) \ | \ \sigma_{x}(1_{B}) \neq 0 \}
\end{align*}
where $|x|$ is the order of $x$. Therefore we prove Theorem \ref{Thm1.3}. This proof is inspired by Okuyama \cite{O}.

\begin{proof}[\textit{Proof of Theorem \ref{Thm1.3}}] We may assume that $D$ is not trivial. We follow two steps. \\

\textbf{Step 1} We show $LL({\bf Z}B) \le \lambda + 1$. \\

Let denote by $Z_{p'}$ the $F$-subspace of ${\bf Z}FG$ spanned by all $p$-regular section sums. For our claim above, it suffices to prove that $({\bf JZ}FG)^{\lambda} 1_{B} \subseteq Z_{p'}$ since $Z_{p'}$ is contained in the socle of $FG$ and thus ${\bf JZ}FG \cdot Z_{p'} = 0$. Take an element $a$ in the left side and write $a = \sum_{g \in G} a_{g}g$ where $a_{g} \in F$. We want to show that $a_{g} = a_{h}$ for all $g, h \in G$, if the $p$-regular parts of them are $G$-conjugate. However, since $a \in {\bf Z}FG$, we need only see that $a_{xy} = a_{y}$ for all $p$-elements $x$ and $p$-regular elements $y$ in $G$ with $xy = yx$. We first suppose $\sigma_{x} (1_{B}) = 0$. Then $\sigma_{x} (a) = 0$ and thus $a_{xy} = a_{y} = 0$ as required. So we may assume that $\sigma_{x} (1_{B}) \neq 0$. Therefore we have 
\begin{align*}
\sigma_{x} (({\bf JZ}FG)^{\lambda} 1_{B}) & \subseteq \sum_{1 \le i \le r} ({\bf JZ}FC_{G}(x))^{\lambda} 1_{b_{i}} \\
    & = \sum_{1 \le i \le r} ({\bf JZ}b_{i})^{\lambda} \\
    & \subseteq \sum_{1 \le i \le r} ({\bf JZ}b_{i})^{\lambda_{x} (|x| - 1)} 
\end{align*}
On the other hand, for each $1 \le i \le r$, $\tau_{x} (({\bf JZ}b_{i})^{\lambda_{x}}) \subseteq ({\bf JZ}\bar{b_{i}})^{\lambda_{x}} = 0$ and hence $({\bf JZ}b_{i})^{\lambda_{x}} \subseteq \kernel \tau_{x}$. 
Since $\kernel \tau_{x} = (x-1)FC_{G}(x)$, we conclude 
\begin{align*}
({\bf JZ}b_{i})^{\lambda_{x}(|x| - 1)} & \subseteq \{ (x-1)FC_{G}(x) \}^{|x| - 1} \\
& = (x-1)^{|x|-1} FC_{G}(x) \\
& = (1 + x + \dots + x^{|x|-1}) FC_{G}(x).
\end{align*}
 
 Thereby we have $\sigma_{x} (a) \in (1 + x + \dots + x^{|x|-1})FC_{G}(x)$, $x\sigma_{x}(a) = \sigma_{x}(a)$ and thus $a_{xy} = a_{y}$ as claimed. \\
 
 \textbf{Step 2} We show $\lambda + 1 \le p^{d} - p^{d - \varepsilon} + 1$. \\
 
 We fix a $p$-element $x$ in $G$ of order $p^{\varepsilon_{1}}$ and block $b$ of $FC_{G}(x)$ associated to $\lambda$. Namely, $\lambda = LL({\bf Z}\bar{b}) (p^{\varepsilon_{1}} - 1)$. We remark that $0 < \varepsilon_{1}$ when $D$ is not trivial. Let $D_{1}$ be a defect group of $b$ of order $p^{d_{1}}$. Then $D_{1}$ is contained in $D$ up to $G$-conjugate, $\varepsilon_{1} \le \varepsilon$ and we may assume the defect group of $\bar{b}$ is $\bar{D_{1}} = D_{1} / \langle x \rangle$ of order $p^{d_{1} - \varepsilon_{1}}$. Therefore we have 
 \begin{align*}
 LL({\bf Z}\bar{b}) (p^{\varepsilon_{1}} - 1) + 1 & \le p^{d_{1} - \varepsilon_{1}}(p^{\varepsilon_{1}} - 1) + 1 \\
                     & \le p^{d - \varepsilon_{1}}(p^{\varepsilon_{1}} - 1) + 1 \\
                     & = p^{d} - p^{d - \varepsilon_{1}} + 1\\
                     & \le p^{d} - p^{d - \varepsilon} + 1.
 \end{align*}
 The theorem is completely proved.
 
\end{proof}

\begin{cor} \label{Cor3.1} In the proof of Theorem \ref{Thm1.3}, if $LL({\bf Z}\bar{b}) (p^{\varepsilon_{1}} - 1) + 1 = p^{d} - p^{d - \varepsilon} + 1$, then $D \simeq C_{p^{\varepsilon}} \times C_{p^{d - \varepsilon}}$. In particular, if $LL({\bf Z}B) = p^{d} - p^{d - \varepsilon} + 1$, then $D \simeq C_{p^{\varepsilon}} \times C_{p^{d - \varepsilon}}$.
\end{cor}

\begin{proof}
By the inequality in \textbf{Step 2} in the proof above, we have $\varepsilon_{1} = \varepsilon$ and $d_{1} = d$. Moreover $\bar{D_{1}} = D_{1} / \langle x \rangle$ is cyclic by Theorem \ref{Thm1.2}. Thus, since $\langle x \rangle$ is contained in the center $Z(D_{1})$ of $D_{1}$, $D_{1} /   Z(D_{1})$ is also cyclic. This implies $D_{1}$ is abelian. Thereby, we have $D \simeq D_{1} = \langle x \rangle \times H$ where $H \simeq D_{1} / \langle x \rangle$. 
\end{proof}

\section{Large Loewy lengths}
In this last section, we prove Theorem \ref{Thm1.4} and \ref{Thm1.5}. We remark that the notations given in the proof of Theorem \ref{Thm1.3} will be used throughout this section. 

\begin{proof}[\textit{Proof of Theorem \ref{Thm1.4}}] 
In case $D$ is cyclic, we obtain $D \simeq C_{3}$ and $e(B)=2$ by Proposition \ref{Prop2.3}. So we may assume $D$ is not cyclic and hence $\varepsilon < d$. Then, since $LL({\bf Z}B) = p^{d} - 1 \le p^{d} - p^{d - \varepsilon} + 1 < p^{d}$, we have $D \simeq C_{2} \times C_{2^{d-1}}$ by Corollary  \ref{Cor3.1}. Furthermore, $d=2$ and (2) holds by Proposition \ref{Prop2.5} as claimed. 

\end{proof}

\begin{proof}[\textit{Proof of Theorem \ref{Thm1.5}}]
As same reason to the proof of Theorem \ref{Thm1.4}, we may assume $D$ is not cyclic, $\varepsilon < d$ and $LL({\bf Z}B) = p^{d} - 2 \le LL({\bf Z}\bar{b})(p^{\varepsilon_{1}}-1)+1 \le p^{d} - p^{d - \varepsilon} + 1 \le p^{d} - 1$. \\

\textbf{Case 1} $LL({\bf Z}B) = p^{d} - 2 = p^{d} - p^{d - \varepsilon} + 1$. \\
By Corollary \ref{Cor3.1}, we have $D \simeq C_{3} \times C_{3^{d-1}}$. However, $LL({\bf Z}B) \neq p^{d} - 2$ in this case by Proposition \ref{Prop2.6}. \\

\textbf{Case 2} $LL({\bf Z}\bar{b})(p^{\varepsilon_{1}} - 1) + 1 = p^{d} - p^{d - \varepsilon} + 1 = p^{d} - 1$. \\
By Corollary \ref{Cor3.1}, $D \simeq C_{2} \times C_{2^{d-1}}$. Moreover, by Proposition \ref{Prop2.5}, 
(2) or (3) holds. \\

\textbf{Case 3} $LL({\bf Z}\bar{b})(p^{\varepsilon_{1}} - 1) + 1 = p^{d} - 2$ and $p^{d} - p^{d - \varepsilon} + 1 = p^{d} - 1$. \\

We obtain $p=2$, $d - \varepsilon = 1$ and $LL({\bf Z}\bar{b}) = \frac{2^d - 3}{2^{\varepsilon_{1}} - 1}$. Since $LL({\bf Z}\bar{b}) \le \bar{D_{1}} \le 2^{d- \varepsilon_{1}}$, $d- \varepsilon_{1}=1$ (remark that $0 < d - \varepsilon \le d - \varepsilon_{1}$) and so $LL({\bf Z}\bar{b}) = \frac{2^{\varepsilon_{1}+1} - 3}{2^{\varepsilon_{1}} - 1} = 1$ or $2$. Thus we have $\varepsilon_{1}=1$ and $d=2$. In this case, (2) or (3) holds by Proposition \ref{Prop2.5}.

\end{proof}

We omit the proof of Theorem \ref{Thm1.6} since we can prove it by similar way to Theorem \ref{Thm1.4} and \ref{Thm1.5}.

\subsection*{Acknowledgment}
The author would like to thank Shigeo Koshitani and Taro Sakurai for useful comments and helpful discussions. The author also would like to thank Benjamin Sambale for pointing out an error in Theorem \ref{Thm1.5} in the first version of this paper.

\end{document}